\documentclass{template}
\usepackage{authblk}

\begin{document}

\title{A Cheeger Inequality for Small Set Expansion}





\author[1]{Akhil Jalan}
\affil[1]{Department of Computer Science, UT Austin}
\affil[1]{\textit{akhiljalan@utexas.edu}}



\date{\today}

\hypersetup{linkcolor={blue},citecolor={blue},urlcolor={blue}}

\bibliographystyle{alpha}

\maketitle

\begin{abstract}

The discrete Cheeger inequality, due to Alon and Milman \cite{alon-milman-cheeger-ineq}, is an indispensable tool for converting the combinatorial condition of graph expansion to an algebraic condition on the eigenvalues of the graph adjacency matrix. We prove a generalization of Cheeger's inqeuality, giving an algebraic condition equivalent to small set expansion. This algebraic condition is the $p \to q$ hypercontractivity of the top eigenspace for the graph adjacency matrix, for any $2 \leq p < q \leq \infty$. Our result generalizes a theorem of \cite{barak-et-al-2012} to the low small set expansion regime, and has a dramatically simpler proof; this answers a question of Barak \cite{barak2014sum}. 
\end{abstract}

{
  \hypersetup{linkcolor=blue}
  \tableofcontents
}
\pagebreak

\section{Introduction}

The discrete Cheeger inequality, due to Alon and Milman, proves that a combinatorial property of graphs (expansion) is equivalent to an algebraic one (spectral gap) \cite{alon-milman-cheeger-ineq}. Since its discovery, Cheeger's inequality has become fundamental in computer science and discrete mathematics; its applications are far too numerous to cite and we refer the reader to the excellent survey \cite{hoory2006expander}.

Small-set expansion (SSE) is a generalization of edge expansion that requires expansion only for sufficiently small sets of vertices. A long line of work reveals deep connections between small-set expansion and the Unique Games Conjecture \cite{raghavendra-steurer-2010, rst-reductions-2010, arora-barak-steurer, steurer-2010,rst-isoperimetric-2010, shortcode, bafna-2021,hopkins2020high}. In particular, small-set expansion for pseudorandom subsets of the Grassman graph was key to the breakthrough proof of the 2-to-2 Games Conjecture \cite{khot-minzer-safra-2018}. 

As part of their work on Unique Games, Barak et al. generalized Cheeger's inequality to small-set expanders \cite{barak-et-al-2012}. Their equivalence is between small-set expansion (a combinatorial property) and {\em hypercontractivity} of the high-eigenvalues eigenspace of the graph (an algebraic property). As a consequence, they immediately obtain a reduction from the Small Set Expansion problem of \cite{rst-reductions-2010} (and hence, from Unique Games) to the (gap) problem of deciding the $2 \to 4$ norm of a graph's top eigenspaces. This gives a new route towards resolving the Small Set Expansion Hypothesis and the Unique Games Conjecture. 

However, their result only holds if small sets are assumed to have very high expansion; for graphs with small-set expansion of, e.g., $0.01$, it does not guarantee any hypercontractivity. 

We prove a more general result that holds in this {\em low-expansion regime}. Our proof is radically simple; it sidesteps the difficulties of \cite{barak-et-al-2012} by reducing $2 \to 4$ hypercontractivity to an equivalence between $4/3 \to 2$ hypercontractivity and small set expansion; the latter is much easier to prove with standard rounding arguments.  

\begin{thrm}[Main Theorem, Informal]\label{main-theorem-informal}
Let $G = (V, E)$ be a finite, undirected, degree-regular graph, and $\delta, \eps \in (0, 1/4)$. Suppose $G$ is such that sets of density up to $\delta$ have expansion at least $\eps$. 

Then, let $V_{1 - \eps}$ denote the direct sum of eigenspaces of $G$ with eigenvalue at least $1 - \eps$. For any $2 \leq p \leq q \leq \infty$, the $p \to q$ norm of the eigenspace of $G$, $\norm V_{1-\eps} \norm_{p \to q} \leq \frac{1}{\sqrt{\delta}}$. 
\end{thrm}

Our main theorem describes the ``hard direction'' (combinatorial to algebraic) of the equivalence. Combined with the already known easy direction (algebraic to combinatorial - see Prop \ref{easy-direction}), this proves equivalence of small-set expansion and top eigenspace hypercontractivity. Moreover, our simplified proof technique also works in the high-expansion regime - see Sec \ref{high-expansion-section}. 

\textbf{Proof overview}: Our proof begins by proving that small-set expansion implies hypercontractivity for the $p \to 2$ norm, for $p \leq 2$, using a Local Cheeger bound. This argument is standard \cite{barak2014sum}. 
However, our hypercontractivity result is in the wrong regime; for applications to Unique Games, we are interested in $2 \to q$ norm for $q > 2$, rather than $p \to 2$ for $p < 2$. 

We can fix this with another known result, which we refer to as {\em \holder duality}. It states if $p^*, q^*$ denote the \holder duals of $p, q$ respectively (meaning $\frac 1 p + \frac{1}{p^*} = 1$, and likewise for $q, q^*$), then for a linear operator $A: \RR^n \to \RR^n$ we have $\norm A \norm_{p \to q} = \norm A^T \norm_{q^* \to p^*}$. 

Since $P_{1-\eps}$ is self-adjoint (it is a projection), it follows that  $\norm P_{1-\eps} \norm_{p \to 2} = \norm P_{1-\eps}^T \norm_{2 \to p^*} = \norm P_{1-\eps}\norm_{2 \to p^*}$. Since that $p \leq 2$ implies $p^* \geq 2$, we are done.  

\textbf{Organization}: In section 2 we provide some preliminaries on norms, hypercontractivity, and small-set expansion. Then in section 3, we prove our main Theorem \ref{main-thrm-full}. Finally, we prove equivalence in the high-expansion setting in Sec \ref{high-expansion-section}. 

\section{Preliminaries}

Throughout this note, let $G = (V, E)$ be an undirected, degree-regular graph. We follow the conventions of $n \defeq \abs{V}, m \defeq \abs{E}$. $A$ will always denote the (normalized) adjacency matrix of $G$. 

\subsection{Hypercontractivity}

\begin{defn}[Norms]
Let $p \geq 1$, $n \geq 1$ be integers, and $x \in \RR^n$. The $\ell_p$ norm of $x$ is $\norm x \norm_{\ell_p} \defeq (\sum\limits_{i=1}^{n} x_i^p)^{1/p}$. The $L_p$ norm of $x$ is $\norm x \norm_{L_p} \defeq (\EE\limits_{i \in [n]} x_i^p)^{1/p}$. 
\end{defn}

Unless explicitly noted otherwise, $\norm x \norm_p$ will denote $\norm x \norm_{L_p}$ throughout this note. 

\begin{defn}[$p \to q$ norm]
Let $p, q \geq 1$. The $p \to q$ norm of any $x \in \RR^n$ is 
$$\norm x \norm_{p \to q} \defeq \frac{\norm x \norm_q}{\norm x \norm_p}$$

If $S \subseteq \RR^n$ then 
$$\norm S \norm_{p \to q} \defeq \sup\limits_{x \in S} \norm x \norm_{p \to q}$$

Finally, if $T: \RR^n \to \RR^n$ is a function, then 
$$\norm T \norm_{p \to q} \defeq \sup\limits_{x \in \RR^n} \frac{\norm Tx \norm_q}{\norm x \norm_p}$$
\end{defn}

In particular, we will be interested in studying the $p \to q$ norms of the eigenspaces of graphs. If an operator or set has $p \to q$ norm at most $1$ for $p < q$, it is \textit{hypercontractive}. 

\begin{defn}
Let $p > q \geq 1$. Let $S \subset \RR^n$. If $\norm S \norm_{p \to q} < 1$, then $S$ is hypercontractive. Similarly if $T: \RR^n \to \RR^n$ and  $\norm T \norm_{p \to q} < 1$, then $T$ is hypercontractive. 
\end{defn}

We will abuse terminology and refer to any operator or set with bounded $p \to q$ norm as hypercontractive. 

\begin{defn}[Top eigenspace]
Let $G = (V, E)$ be an undirected, degree-regular graph and $\lambda \in [-1, 1]$. Let $A$ denote its normalized adjacency matrix. Then $V_\lambda$ is the direct sum of the eigenspaces of $A$ with eigenvalue $\geq \lambda$. 
\end{defn}

\begin{defn}[Projector]
Let $G = (V, E)$ be an undirected, degree-regular graph and $\lambda \in [-1, 1]$. Then ${P_\lambda(G): \RR^V \to \RR^V}$ is the projection operator onto $V_\lambda$. 
\end{defn}

We will write $P_\lambda$ when the graph is clear from context. Now, we are ready to define eigenspace hypercontractivity. 

\begin{defn}
Let $G = (V, E)$ be an undirected, degree-regular graph and $\lambda \in [-1, 1]$. Let $1 \leq p < q$. Let $C > 0$ be a constant. We say that $G$ is $(p, q, \lambda, C)$-hypercontractive if 
$$\norm P_\lambda(G) \norm_{p \to q} \leq C$$
\end{defn}

We can equivalently define hypercontractivity of a graph in terms of the $p \to q$ norm of the top eigenspace, rather than the projector - see Lemma \ref{hyp-equiv-projector-subspace}. 

\subsection{Subspace vs Projector Hypercontractivity}

In this section we will show that for the eigenspace projector $P_\lambda$, We will show that the $p \to q$ norms of the eigenspace $V_\lambda \defeq im(P_\lambda)$ and of the projector are equal; we will use both interchangeably. 


\begin{lemma}\label{projection-is-orthogonal}
Let $G$ be an undirected, $d$-regular graph. Then for any $\lambda \in [-1, 1]$, the projector $P_\lambda$ is an orthogonal projection.   
\end{lemma}

\begin{proof}
The matrix $P_\lambda$ is a projector by definiton. It is an orthogonal projection iff its range is orthogonal to its null-space. This follows from the fact that the (distinct) eigenspaces of $G$ are all orthogonal to one another. 

That is, let $A$ be the adjacency matrix of $G$. Then since $G$ is undirected, $A = A^T$, and so $Av = \mu v$ and $Aw = \lambda w$ implies that $\mu \la v, w \ra = \la Av, w \ra = \la v, Aw \ra = \lambda \la v, w \ra$. Therefore $\mu \neq \lambda$ implies that $v \perp w$. 
\end{proof}

\begin{cor}
Using the same notation, $P_\lambda = P_\lambda^T$. 
\end{cor}

\begin{lemma}[Equivalence of projector and subspace hypercontractivity]\label{hyp-equiv-projector-subspace}
For any $\lambda \in [-1, 1]$, $\norm P_\lambda \norm_{p \to q} = \norm V_\lambda \norm_{p \to q}$. 
\end{lemma}

\begin{proof}
First, if $v \in V_\lambda$ then $P_\lambda v = v$, so $\norm P_\lambda \norm_{p \to q} \geq \norm V_\lambda \norm_{p \to q}$. 

Conversely, for any $v \in \RR^V$, notice that $\norm P_\lambda v \norm_p \leq \norm v \norm_p$ for any $p \geq 1$ (this follows from the fact that $P_\lambda$ has spectrum inside $\{0, 1\}$). Therefore $\norm P_\lambda \norm_{p \to q} \leq \frac{\norm P_\lambda v \norm_q}{\norm v \norm_p} \leq \frac{\norm P_\lambda v \norm_q}{\norm P_\lambda v \norm_p} \leq \norm V_\lambda \norm_{p \to q}$. 
\end{proof}

\subsection{Expansion and Small Set Expansion}

Next, we define the (edge) expansion of sets. Given a graph $G$ and set of vertices $S$, the edge expansion quantifies the probability that a random step beginning in $S$ will reach a vertex outside of $S$.

\begin{defn}[Expansion]
Let $G = (V, E)$ be an undirected, degree-regular graph and $S \subset V$. Then the expansion of $S$ is $$\Phi(S) \defeq \PP_{v \in S, w \sim v}[w \not \in S]$$

We also denote the non-expansion of $S$ as $\barr{\Phi(S)} \defeq 1 - \Phi(S)$. 
\end{defn}

Finally, let $\mu(S) \defeq \frac{\abs{S}}{n}$ denote the \textit{density}\footnote{There is a more general definition of density for graphs that are not degree-regular, but we will not need it here.} of a set of vertices. 

\begin{defn}[Small set expander]
Let $G = (V, E)$ be an undirected, degree-regular graph and $\delta > 0$. The $\delta$-expansion of $G$ is 
$$\Phi(\delta) \defeq \inf\limits_{S \subset V: \mu(S) \leq \delta} \Phi(S)$$
\end{defn}

\section{Equivalence of Eigenspace Hypercontractivity and Small Set Expansion}

\subsection{The easy direction: Hypercontractivity implies small-set expansion}

It follows from \holder's inequality that hypercontractivity of $P_{\lambda}$ implies small-set expansion of $G$. 

\begin{prop}[\cite{moshkovitz-2021} 1.10]\label{easy-direction}
Let $\eps > 0, p, q \geq 1$. Then if for a graph $G = (V, E)$ the projector $P_\eps$ has bounded $p \to q$ norm, then for all $S \subset V$,
$$\barr{\Phi(S)} \leq \norm P_\eps \norm_{p \to q} \mu(S)^{1/p - 1/q} + \eps$$
\end{prop}

\subsection{SSE implies \texorpdfstring{$p \to q$}{p to q} Hypercontractivity for \texorpdfstring{$1 \leq p < q \leq 2$}{1 to 2}}

In this section, we will show that small-set expansion (SSE) implies $p \to q$ hypercontractivity for $1 \leq p \leq q \leq 2$. The following Lemma is standard. 

\begin{lemma}\label{inner-prod-2-norm}
If $v \in V_{\lambda}$ then $\la v, Av \ra \geq \lambda \norm v \norm_2^2$. 
\end{lemma}

\begin{proof}
By the spectral theorem $A = \sum_{i} \lambda_i u_i u_i^T$ for some orthonormal eigenbasis $\{u_i\}$. Let $v = \sum_i \alpha_i u_i$ in this basis. Notice that since $v \in V_\lambda$ that $\alpha_i = 0$ for $i$ such that $\lambda_i < \lambda$. Therefore, 

\begin{align}
\la v, Av \ra = \la \sum_i \alpha_i u_i, \sum_j \alpha_j \lambda_j u_j \ra 
= \sum_{i, j} \alpha_i \alpha_j \lambda_j \la u_i, u_j \ra 
= \frac 1 n \sum_i \alpha_i^2 \lambda_i 
\geq \lambda \cdot \frac 1 n \sum_i \alpha_i^2 
= \lambda \norm v \norm_2^2 
\end{align}
\end{proof}


Next, we prove equivalence of SSE to eigenspace  $p \to q$ hypercontractivity, but only for the regime of $1 \leq p < q \leq 2$. This is the most technical part of our argument, so we discuss it here. 

{\bf The proof strategy.} Proofs of Cheeger's inequality typically use rounding to prove a contradiction. Specifically, assuming that $G$ has small spectral gap, there will exist $v \in \RR^V$ such that $Av = \lambda v$ for $\lambda$ near $1$. One can ``round'' $v$ to a vector $v^\prime \in \bb^V$ such that the corresponding set has low expansion, proving a contradiction. 

In our setting, we want to round a vector $v \in \RR^V$ to a $v^\prime \in \bb^V$ corresponding to a {\em small} set. The tool for this rounding is the Local Cheeger Bound (LCB).

More precisely, let $v^2 \in \RR^V$ denote $v$ with each entry squared. The Local Cheeger Bound gives a level set of $S$ of $v^2$ with certain properties that can prove an upper bound on $\Phi(S)$, showing a contradiction. 

The bound on $\Phi(S)$ depends on two quantities. First, we require that a lower bound on $\frac{\la v, Av \ra}{\norm v \norm_2^2}$. Informally, this means that $v$ has low ``expansion.'' To see this, notice that if $v = 1_B$ for some $B \subset V$, then $1 - \Phi(B) = \frac{\la v, Av \ra}{\norm v \norm_2^2}$. Therefore, even though $v$ does not correspond to a set in general, a lower bound on the ratio $\frac{\la v, Av \ra}{\norm v \norm_2^2}$ means that we would expect level sets of $v$ to have high non-expansion. 

Second, we require that $\delta \norm v \norm_2^2 \geq \norm v \norm_1^2$. After scaling, $v$ is a probability distribution on $V$. Therefore, this condition means that $v$ corresponds to a distribution on the vertices with high collision probability. Specifically, $\delta \norm v \norm_2^2 \geq \norm v \norm_1^2$ implies that $\EE_{i \in V} v_i^2 = \norm v \norm_2^2 \geq \frac{1}{\delta n}$. 

Let us formally state the version of the Local Cheeger Bound we need. 

\begin{lemma}[\cite{steurer-2010} 6.1]\label{Steurer-LCB}
For all $v \in \RR^V$ there exists a level set $S$ of $v^2$ such that if $\mu(S) \leq \delta$, $S$ has expansion

$$\Phi(S) \leq \frac{\sqrt{1 - \frac{\la v, Av \ra^2}{\norm v \norm_2^4}}}{1 - \frac{\norm v \norm_1^2}{\delta \norm v \norm_2^2}}$$
\end{lemma}

We are ready to prove the main claim of this section. In essence, we will assume that $P_{1-\eps}$ has large $p \to q$ norm for contradiction, and use this to obtain a distribution on $V$ with high collision probability (e.g. high $1 \to 2$ norm). Then, we will use the Local Cheeger Bound to round this distribution to a small set with very low expansion. 

\begin{prop}\label{1-to-2}
Let $G = (V, E)$ be a finite, undirected degree-regular graph such that for all $S \subset V$, $\mu(S) \leq 4\delta \rarr \Phi(S) \geq 2\sqrt{\eps}$.

Let $\abs{V} = n$, and $\eps > 0$. Then, for any $1 \leq p \leq q \leq 2$, it follows that $\norm P_{1 - \eps}\norm_{p \to q} < \frac{1}{\sqrt{\delta}}$. 
\end{prop}

\begin{proof}
Suppose towards a contradiction that $\norm P_{1 - \eps}\norm_{p \to q} < \frac{1}{\sqrt{\delta}}$. Then there there exists $w \in V_{1 - \eps}$ such that $\delta^{q/2} \norm w \norm_{q}^{q} \geq \norm w \norm_p^{q}$.

For any $s \geq r \geq 1$, a convexity argument shows that for all $v \in \RR^n$, $\norm v \norm_{L^s} \geq \norm v \norm_{L^r}$. Therefore,  

$$\delta^{q/2} \norm w \norm_2^q \geq \delta^{q/2} \norm w \norm_{q}^{q} \geq \norm w \norm_p^{q} \geq \norm w \norm_1^q$$

Taking the $(2/q)^{th}$ power gives us that $\delta \norm w \norm_2^2 \geq \norm w \norm_1^2$.



Next, normalize $w$ so that $\sum_i \abs{w_i} = 1$. Denote the normalized vector as $z$. 

Next, let $A$ be the normalized adjacency matrix. By Lemma \ref{inner-prod-2-norm}, we have that $z \in V_{1 - \eps}$, so $\la z, Az \ra \geq (1 - \eps) \norm z \norm_2^2$. 

Next, let $z^2$ denote $z$ with all entries squared. By Lemma \ref{Steurer-LCB} ther is a level set $S$ of $z^{2}$ such that $\mu(S) \leq 4 \delta$ and 

$$\Phi(S) \leq \frac{\sqrt{1 - \frac{\la z, Az \ra^2}{\norm z \norm_2^4}}}{1 - \norm z \norm_1^2 / (4 \cdot \delta) \norm z \norm_2^2}$$

We know that $\la z, Az \ra^2 \geq (1 - \eps)^2 \norm z \norm_2^4$. So the numerator simplifies to $\sqrt{1 - \frac{\la z, Az \ra^2}{\norm z \norm_2^4}} \leq \sqrt{1 - (1 - \eps)^2} = \sqrt{2\eps - \eps^2}$.

For the denominator, we will use the fact that $\delta \norm z \norm_2^2 \geq \norm z \norm_1^2$. 

Simplifying, we obtain that 

$$\Phi(S) \leq \frac{\sqrt{2\eps - \eps^2}}{1 - \frac 1 4} < 2 \sqrt{\eps}$$

But we assumed that $\Phi(S) \geq 2\sqrt{\eps}$. Contradiction.
\end{proof}

\subsection{Proof of Main Theorem}\label{main-low-expansion}

Our main theorem follows from Prop \ref{1-to-2}, combined with the following simple Lemma (``\holder duality.'')

\begin{lemma}[\cite{bhattiprolu-2018} 2.4]\label{holder-duality}
Let $A: \RR^n \to \RR^n$ be a linear operator and $1 \leq p, q \leq \infty$. Let $p^*$ be the unique real such that $\frac 1 p + \frac{1}{p^*} = 1$ (its \holder dual), and let $q^*$ be the \holder dual of $q$. 

Then, $$\norm A \norm_{L^p \to L^q} = \norm A^T \norm_{L^{q^*} \to L^{p^*}}$$ 
\end{lemma}

\begin{thrm}[Main Theorem, Formal]\label{main-thrm-full}
Let $G = (V, E)$ be a finite, undirected, degree-regular graph, and $\delta, \eps \in (0, 1/4)$. Suppose $\Phi(4\delta) \geq 2\sqrt{\eps}$.

Then for any $2 \leq p < q \leq \infty$, we claim that $\norm P_{1 - \eps}\norm_{p \to q} < \frac{1}{\delta^{1/2}}$. 
\end{thrm}

\begin{proof}
Let $p^*, q^*$ denote the \holder duals of $p, q$ respectively. Since $1 \leq q^* < p^* \leq 2$, by Prop \ref{1-to-2}, it follows that $\norm P_{1 - \eps}\norm_{q^* \to p^*} < \frac{1}{\sqrt{\delta}}$.

Then, by Lemma \ref{holder-duality}, $\norm P_{1 - \eps}^T \norm_{(p^*)^* \to (q^*)^*} < \frac{1}{\sqrt{\delta}}$. Since $(p^*)^* = p$ and $(q^*)^* = q$, this simplifies to $\norm P_{1 - \eps}^T \norm_{p \to q} < \frac{1}{\sqrt{\delta}}$.

Finally, by Lemma \ref{projection-is-orthogonal}, $P_{1-\eps}^T = P_{1-\eps}$. So we conclude $\norm P_{1 - \eps} \norm_{p \to q} < \frac{1}{\sqrt{\delta}}$. 
\end{proof}

\subsection{The High-Expansion Regime}\label{high-expansion-section}

Previous work proves that if $G$ is a graph where small sets have very high expansion, then the top eigenspace is hypercontractive \cite{barak-et-al-2012} . Our simplified proof technique works in this high-expansion regime as well, if we use a different version of the Local Cheeger Bound. 

\begin{lemma}[LCB for High-Expansion, \cite{steurer-2010} 2.1]\label{lcb-high-expansion}
Let $G$ be a regular graph and $z \in \RR^V$ be such that $\norm Az \norm_2^2 \geq \eps \norm z \norm_2^2$ and $\norm z \norm_1^2 \leq \delta \norm z \norm_2^2$. Then there exists a level set $S$ of $(Az + z)^2 \in \RR^V$, such that $\Phi(S) \leq 1 - C \eps^2$. 
\end{lemma}

It suffices to take $C = 100$ above. 

Now, we repeat our argument but use this version of the LCB. 

\begin{prop}
Let $G = (V, E)$ be a finite undirected degree-regular graph, and $\delta, \eps > 0$. Suppose $\Phi(\delta) > 1 - 100\eps^2$. 

Then for any $2 \leq p \leq q \leq \infty$, we claim that $\norm P_{\sqrt{\eps}}\norm_{p \to q} < \frac{1}{\delta^{1/2}}$. 
\end{prop}

\begin{proof}
Following the proof of Theorem \ref{main-thrm-full}, it suffices to show that $\norm P_{\sqrt{\eps}}\norm_{q^* \to p^*} < \frac{1}{\delta^{1/2}}$. 

Assume towards a contradiction that there exists $w \in V_{\sqrt{\eps}}$ such that $\norm w \norm_{q^* \to p^*} \geq \frac{1}{\sqrt{\delta}}$. 

First, since $w \in V_{\sqrt{\eps}}$ it follows that $\norm Aw \norm_2^2 \geq \eps \norm w \norm_2^2$. 

Second, from $\norm w \norm_{q^* \to p^*} \geq \frac{1}{\sqrt{\delta}}$, and monotonicity of $L_p$ norms, and the fact that $1 \leq q^* \leq p^* \leq 2$, it follows that $\norm w \norm_1^2 \leq \delta \norm w \norm_2^2$. 

From these two facts about $w$, it follows from Lemma \ref{lcb-high-expansion} that there exists $S \subset V$ such that $\mu(S) \leq \delta$ and $\Phi(S) \leq 1 - 100 \eps^2$. Contradiction. 
\end{proof}

\section{Acknowledgments}

The author thanks Dana Moshkovitz for introducing him to the problem and collaboration on an earlier version of this manuscript. This material is based upon work supported by the National Science Foundation under grants number 1218547 and 1678712.

\bibliography{refs.bib}

\end{document}